\documentclass[12pt,reqno]{article}
\usepackage{fullpage,hyperref,amsmath, amssymb, amsthm, url,enumerate,graphicx,rotating,tikz,relsize}
\usepackage{scaffold}  

\newtheorem{theorem}{Theorem}[section]
\newtheorem{proposition}[theorem]{Proposition}
\newtheorem{lemma}[theorem]{Lemma}

\newtheorem{definition}[theorem]{Definition}
\newtheorem{example}[theorem]{Example}

\newtheorem{remark}[theorem]{Remark}

\newcommand\cref[1]{Corollary~\ref{cor:#1}}

\newcommand\sqr[2]{{\vbox{\hrule height.#2pt
    \hbox{\vrule width.#2pt height#1pt \kern#1pt
        \vrule width.#2pt}\hrule height.#2pt}}}
\renewcommand\qed{%
        \ifmmode\eqno\sqr53
        \else\nolinebreak\ \hfill\sqr53\medbreak\fi}

\numberwithin{equation}{section}
\newcommand{\Rrel}[1]{\overset{#1}{\sim}}

\newcommand{\re}{{\mathbb R}}
\newcommand{\cx}{{\mathbb C}}

\newcommand{\BMA}{{\mathbb A}}

\newcommand{\cR}{{\mathcal R}}

\newcommand{\cS}{{\mathcal S}}

\newcommand{\scafs}{{\mathsf{s}}}

\newcommand{\cU}{{\mathcal U}}

\newcommand{\cP}{{\mathcal P}}

\newcommand{\Mat}{\mathsf{Mat}}
\newcommand{\sS}{\mathsf{S}}
\newcommand{\tr}{\mathsf{tr}}

\DeclareMathOperator\spn{span}

\makeatletter  
\newtheorem*{rep@theorem}{\rep@title}
\newcommand{\newreptheorem}[2]{%
\newenvironment{rep#1}[1]{%
 \def\rep@title{#2 \ref{##1}}%
 \begin{rep@theorem}}%
 {\end{rep@theorem}}}
\makeatother    

\newcommand{\inlineEG}[4]{
\begin{tikzpicture}[baseline=(base),black,node distance=0.5cm,  
rootnode/.style={draw, circle,  fill=red, inner sep=1.8pt},
hollownode/.style={  draw,  circle,  fill=white,  inner sep=1.8pt},
  every loop/.style={min distance=40pt,in=-30,out=90,looseness=20}]
\node[hollownode] (w) at (180:#4) {};
\node[hollownode] (ne) at (60:#4) {};
\node[hollownode] (se) at (300:#4) {};
 \path[-,thick] (se) edge node[left] {$\scriptscriptstyle{A_{#1}}$}   (ne);
 \path[-,thick] (w) edge node[above] {$\scriptscriptstyle{A_{#2}}$}   (ne);
 \path[-,thick] (w) edge node[below] {$\scriptscriptstyle{A_{#3}}$}   (se);
 \end{tikzpicture}}
\newcommand{\inlineBEG}[4]{
\begin{tikzpicture}[baseline=(base),black,node distance=0.5cm,  
rootnode/.style={draw, circle,  fill=red, inner sep=1.8pt},
hollownode/.style={  draw,  circle,  fill=white,  inner sep=1.8pt},
  every loop/.style={min distance=40pt,in=-30,out=90,looseness=20}]
\node[hollownode] (w) at (180:#4) {};
\node[hollownode] (ne) at (60:#4) {};
\node[hollownode] (se) at (300:#4) {};
 \path[-,thick] (se) edge node[left] {$\scriptscriptstyle{A_{#1}}$}   (ne);
 \path[-,thick] (w) edge node[above] {$\scriptscriptstyle{A_{#2}}$}   (ne);
 \path[-,thick] (w) edge node[below] {$\scriptscriptstyle{A_{#3}}$}   (se);
 \path[-,thick] (se) edge[bend right=10] node[right] {$\scriptscriptstyle{A^2}$}   (ne);
 \path[-,thick] (w) edge[bend left=50] node[above] {$\scriptscriptstyle{A^2}$}   (ne);
 \path[-,thick] (w) edge[bend right=50] node[below] {$\scriptscriptstyle{A^2}$}   (se); 
 \end{tikzpicture}}
 \newcommand{\inlineWhiskeredDelta}[3]{ 
\begin{tikzpicture}[baseline=(A1),black,node distance=0.5cm,  
solidvert/.style={draw, circle,  fill=red, inner sep=1.8pt},
hollowvert/.style={  draw,  circle,  fill=white,  inner sep=1.8pt},
  every loop/.style={min distance=40pt,in=-30,out=90,looseness=20}]
\def\sc {0.3}
\node[hollowvert] (A1) at (-1*\sc,0*\sc) {};
\node[hollowvert] (A2) at (2*\sc,0*\sc) {};
\node[hollowvert] (A3) at (0.5*\sc,1.5*1.732*\sc) {};
  \path[-] (A1)  edge [below] node  [] {\scriptsize #3} (A2);
  \path[-] (A2)  edge [right] node  [] {\scriptsize #1} (A3);
  \path[-] (A3)  edge [left] node   [] {\scriptsize #2} (A1);
\draw (A1) --  (-1.9*\sc,0*\sc);
\draw (A1) --  (-1.6*\sc,-0.5*\sc);
\draw (A1) --  (-1.3*\sc,-0.8*\sc);
\draw (A2) --  (2.9*\sc,0*\sc);
\draw (A2) --  (2.6*\sc,-0.5*\sc);
\draw (A2) --  (2.3*\sc,-0.8*\sc);
\draw (A3) --  (-0.2*\sc,3.4*\sc);
\draw (A3) --  (0.5*\sc,3.5*\sc);
\draw (A3) --  (1.2*\sc,3.4*\sc);
\end{tikzpicture}}
 \newcommand{\inlineWhiskeredWye}[3]{ 
 \begin{tikzpicture}[baseline=(B1),black,node distance=0.5cm,  
solidvert/.style={draw, circle,  fill=red, inner sep=1.8pt},
hollowvert/.style={  draw,  circle,  fill=white,  inner sep=1.8pt},
  every loop/.style={min distance=40pt,in=-30,out=90,looseness=20}]
\def\sc {0.3}
\node[hollowvert] (B1) at (-1*\sc,0*\sc) {};
\node[hollowvert] (B2) at (2*\sc,0*\sc) {};
\node[hollowvert] (B3) at (0.5*\sc,1.5*1.732*\sc) {};
\node[hollowvert] (B4) at (0.5*\sc,1.5*.732*\sc) {};
  \path[-] (B1)  edge [left] node  [pos=0.85]  {\scriptsize #2} (B4);
  \path[-] (B2)  edge [right] node  [pos=0.85] {\scriptsize #3} (B4);
  \path[-] (B3)  edge [right] node  [pos=0.5] {\scriptsize #1} (B4);
\draw (B1) --  (-1.9*\sc,0*\sc); \draw (B1) --  (-1.6*\sc,-0.5*\sc); \draw (B1) --  (-1.3*\sc,-0.8*\sc); 
\draw (B2) --  (2.9*\sc,0*\sc); \draw (B2) --  (2.6*\sc,-0.5*\sc); \draw (B2) --  (2.3*\sc,-0.8*\sc); 
\draw (B3) --  (-0.2*\sc,3.4*\sc); \draw (B3) --  (0.5*\sc,3.5*\sc); \draw (B3) --  (1.2*\sc,3.4*\sc); 
\end{tikzpicture}}
 
\begin{document}
\thispagestyle{empty}
\setcounter{page}{1}
\title{Quantum isomorphism of graphs from association schemes}
\author{
Ada Chan \\
Department of Mathematics and Statistics \\
York University, Toronto, Canada \\
{\tt ssachan@yorku.ca} \\
William J.~Martin \\
Department of Mathematical Sciences \\
Worcester Polytechnic Institute, 
Worcester, MA USA \\
{\tt  martin@wpi.edu}} 

\date{\today} 
\maketitle

\medskip

\begin{abstract}
We show that any two Hadamard graphs on the same number of vertices are quantum isomorphic. This follows from a more general recipe for showing quantum isomorphism of graphs arising from certain association schemes.
The main result is built from three tools. A remarkable recent result \cite{manrob2} of Man\v{c}inska and Roberson shows that graphs $G$ and $H$ are quantum isomorphic if and only if, for any planar graph $F$, the number of graph homomorphisms from $F$ to $G$ is equal to the number of graph homomorphisms from $F$ to $H$. A generalization of partition functions called ``scaffolds'' \cite{WJMscaff} affords some basic reduction rules such as series-parallel reduction and can be applied to counting homomorphisms. The final tool is the classical theorem of Epifanov showing that any plane graph can be reduced to a single vertex and no edges by extended series-parallel reductions and Delta-Wye transformations. This last sort of transformation is available to us in the case of exactly triply regular association schemes. The paper includes open problems and directions for future research.
\end{abstract}


\noindent {\bf 2020 MSC Subject Codes:} 05E30, 20G42, 15A72, 16S50, 05C83, 05C90.

%
\section{Introduction}
\label{Sec:intro}

{
The pioneering work of Bell \cite{Bell} showing that no hidden variable theory  can fully explain quantum mechanical correlations of spatially separated entangled particles has been experimentally verified; related experimental work of Aspect, Clauser and Zeilinger earned the Nobel Prize in Physics in 2022. Over recent decades, Bell's ideas have been refined and generalized in many directions, with one application being the study of quantum games. In one of the most basic quantum games, two physically separated players who share both a strategy and a quantum state, but are otherwise unable to communicate, are fed classical questions from a referee and respond with classical answers.  Bell-type inequalities identify a gap between the win probability of the optimal classical strategy and the corresponding probability in this quantum version.  A particularly attractive class of quantum games are the graph homomorphism games. These games are intuitive, cast in the familiar language of computer science protocols. They generate new lines of investigation in graph theory; for example, the quantum chromatic number of a graph is sandwiched between the classical chromatic number and the vector chromatic number. And, perhaps surprisingly, the natural definition of the quantum analogue of the automorphism group of a graph happens to be a compact quantum group.
}

This notion of quantum isomorphism of graphs was introduced by Atserias~et al.\ \cite{manrob4} in their study of a non-local graph isomorphism game with two quantum players.  In that paper, the first known construction is given for non-isomorphic graphs that are quantum isomorphic. Further examples were discovered by S.\ Schmidt \cite{schmidt} using Godsil-McKay switching; Schmidt constructs a
family of pairs of strongly regular graphs on 120 vertices that are quantum isomorphic but pairwise non-isomorphic.
(Schmidt also lists several references proposing alternative approaches, but these have led to no new examples so far.) In this paper, we give a different strategy of finding quantum isomorphic but non-isomorphic graphs using association schemes and scaffolds. We show how this approach implies that any two Hadamard graphs on the same number of vertices are quantum isomorphic. The feature we exploit is exact triple regularity, a property 
closely tied to the study of spin models \cite{jaeger}.

Two graphs $G$ and $H$, with adjacency matrices $A_G$ and $A_H$ respectively, are isomorphic if and only if there exists a permutation matrix $P$ such that
$$
    P A_G= A_H P.
$$
Lov\'asz's classical result states that two graphs are isomorphic if and only if they have the same number of graph homomorphisms from any graph \cite{lovasz67}.

{
Let $\mathcal{A}$ be a $C^*$-algebra with unity $\mathbf{1}$ 
and let $\cU=(u_{ij})$  be an $n\times n$ matrix with 
entries in $\mathcal{A}$. We call $\cU$  a 
\emph{quantum permutation matrix} if it satisfies $u_{ij}^*=u_{ij}$ and
$\sum_{k=1}^n u_{ik} u_{jk}= \delta_{i,j} \mathbf{1} =
\sum_{k=1}^n u_{ki} u_{kj}$ for all $1\le i,j\le n$. See, for 
example, \cite{manrob2} for much more on quantum permutation matrices.
Two graphs, $G$ and $H$, are \emph{quantum isomorphic}  if and only if 
there exists a quantum permutation matrix $\cU$ of block size $d$ such that
$$
\cU A_G  =  A_H \cU
$$
where operations are performed in the $\cx$-algebra $\mathcal{A}$.
Note that when $\mathcal{A}=\cx$, $\cP$ is a permutation matrix and $G$ is isomorphic to $H$.}
See \cite[p323]{manrob4} for a pair of non-isomorphic graphs on $24$ vertices that are quantum isomorphic.  Only two families of examples are known of quantum isomorphic but non-isomorphic graphs; the graphs in the first family \cite{manrob4} are constructed based on a reduction from linear binary constraint system games to isomorphism games, and the graphs in the second family \cite{schmidt} are constructed via Godsil-McKay switching on two particular strongly regular graphs  with parameters $(120, 63, 30, 36)$.

Given graphs $F$ and $G$, we use $\hom(F,G)$ to denote the number of graph homomorphisms from $F$ to $G$.
Man\v{c}inska and Roberson give the following  remarkable characterization of quantum isomorphic graphs.
\begin{theorem}[\cite{manrob2}] 
\label{Thm:quantum}
Two graphs, $G$ and $H$, are quantum isomorphic if and only if 
$$\hom(F,G)=\hom(F,H),$$ 
for any planar graph $F$. 
\end{theorem}
They show further that it is undecidable to determine if two graphs are quantum isomorphic.  Equivalently, given any graphs $G$ and $H$, the problem of determining if there exists a planar $F$ such that $\hom(F,G)\neq \hom(F,H)$ is undecidable.  

For $\varphi:V(F) \rightarrow V(G)$, we have
$$
\prod_{\{a,b\} \in E(F)} (A_G)_{\varphi(a),\varphi(b)} = 
\begin{cases}
1 & \text{if $\varphi$ is a graph homomorphism from $F$ to $G$,}\\
0& \text{otherwise.}
\end{cases}
$$
Hence
$$
\hom(F,G)=\sum_{\varphi:V(F) \rightarrow V(G)} \prod_{\{a,b\} \in E(F)} (A_G)_{\varphi(a),\varphi(b)},
$$
which is the scaffold $\sS(F, \emptyset;w )$ on the graph $F$ with no root node and a weight function $w$ that maps every edge of $F$ to the matrix $A_G$.
Please see Section~\ref{Sec:schemes} for some background on scaffolds and association schemes.

In Section~\ref{Sec:homo}, we consider the case where $A_G$ belongs to the Bose-Mesner algebra of an exactly triply regular association scheme.   We apply Epifanov's theorem to express the scaffold $\sS(F, \emptyset;w )$ on any connected planar graph $F$ in terms of the Delta-Wye parameters of the association scheme.
This observation leads to our main result.

\newtheorem*{thm:main}{Theorem \ref{Thm:quantumisographs}}
\begin{thm:main}
Let $(X, \cR)$ and $(Y, \cS)$ be an exactly triply regular  symmetric association schemes that have the same Delta-Wye parameters (Definition \ref{Def:sameDeltaWye}).
Let $G'$ be a graph in $(Y, \cS)$ corresponding to\footnote{See Definition \ref{Def:corresponding} below.} $G$ in $(X,\cR)$.   Then $G$ and $G'$ are quantum isomorphic.
\end{thm:main}

We focus on Hadamard graphs in Section~\ref{Sec:Hadamard}.
In the construction of spin models from Hadamard graphs \cite{nomura}, Nomura computes the Delta-Wye parameters of the association scheme of a Hadamard graph and shows that these parameters depend only on the number of vertices.   
Hence, the association schemes of two Hadamard graphs of the same order have the same Delta-Wye parameters.   From \cite{jaeger} and \cite{nomura}, we see that the association scheme of a Hadamard matrix is exactly triply regular which leads to the following result, which has also been obtained independently by Gromada \cite{Gromada} via entirely different means.

\newtheorem*{thm:Had}{Theorem \ref{Thm:Hadamard}}
\begin{thm:Had}
Any two Hadamard graphs of the same order are quantum isomorphic.
\end{thm:Had}

To our knowledge, the Hadamard graphs of the same order ($\geq 64$) are the first examples of  three or more mutually quantum isomorphic but not isomorphic graphs.
Merchant proved \cite{Merchant} that, provided at least one Hadamard matrix of order $n$ exists, there are at least $2^{2n-16-6\log n}$
inequivalent Hadamard matrices of order $2n$. And McKay \cite{McKay1979} 
proved that, given two Hadamard matrices $H$ and $H'$, if $H'$ is inequivalent to both $H$ and $H^\top$, then $H$ and $H'$ produce non-isomorphic Hadamard graphs. The longstanding Hadamard conjecture, claiming that there exists a Hadamard matrix of order $4m$ for each positive integer $m$ would then imply that there are an exponential number (in $m$) non-isomorphic Hadamard
graphs on $32m$ vertices.

We discuss open problems and directions for future research in Section~\ref{Sec:problems}.

\section{Association schemes and scaffolds}
\label{Sec:schemes}

A symmetric association scheme is a partition of the edges of a complete graph into regular graphs whose adjacency matrices span a vector space closed and commutative under multiplication (a ``Bose-Mesner algebra''). The regularity imposed by the definition, and by extra assumptions such as triple regularity, facilitate counting of pairs, triples and $m$-tuples of vertices forming prescribed configurations. While the idea has been used informally in the community for decades, the concept of a ``scaffold'' was recently introduced in \cite{WJMscaff} to treat these counts of $m$-tuples algebraically so that linear combinations of $m$-vertex counts can be taken and change-of-basis can be locally applied on such configurations. In this section, we introduce these concepts, and set up notation and terminology that we will use as we work with exactly triply regular association schemes later.

\subsection{Definitions and our notation}
\label{Subsec:schemebasics}

A $d$-class \emph{symmetric association scheme} \cite{del,banito,bcn,godsil,DRGsurvey,mtsurvey} is an ordered pair $(X,\cR)$ where 
 $X$ is a nonempty finite set and  $\cR=\{R_0,\ldots,R_d\}$ is a partition of $X\times X$ into non-empty relations satisfying
 \begin{itemize}
 \item $R_0 = \{ (x,x) \mid x\in X\}$ is the identity relation;
 \item  for each $i$, $0 \le i\le d$, we have $R_i^\top=R_i$ where $R_i^\top = \{ (y,x) \mid (x,y)\in R_i \}$; 
 \item there exist \emph{intersection numbers} $p_{ij}^k$, $0\le i,j,k \le d$ satisfying 
 $$| \{ z\in X \mid (x,z) \in R_i, (z,y)\in R_j\}|=p_{ij}^k$$
 whenever $(x,y)\in R_k$.
\end{itemize} 
Note that, since all relations are symmetric, we have $p_{ij}^k = p_{ji}^k$ for all $i,j,k$; all symmetric association schemes are \emph{commutative}.

 For $x,y\in X$ and $0\le i\le d$, we write $x \Rrel{i} y$ to mean $(x,y)\in R_i$.
 
Denote by $\Mat_X(\cx)$ the algebra of all matrices with rows and columns indexed by set $X$ having complex entries. We define \emph{adjacency matrices} (or \emph{Schur idempotents}) of the association scheme,
$A_0,\ldots, A_d \in  \Mat_X(\cx)$ by 
$$(A_i)_{x,y} = \begin{cases} 1 &\text{if $x \Rrel{i} y$,}\\ 0 & \text{otherwise.}\end{cases}$$
These satisfy
 \begin{itemize}
 \item  $\sum_{i=0}^d A_i = J$, the all ones matrix 
 \item $A_i \circ A_j = \delta_{i,j} A_i$ where $\circ$ is the entrywise or \emph{Hadamard/Schur} product;
 \item $A_0 = I$;
 \item for each $i$, $0 \le i\le d$, $A_i^\top =A_i$;
 \item $\BMA= \spn_\re(\{A_0,\ldots,A_d\})$ is closed and commutative under matrix multiplication: there exist $p_{ij}^k$, $0\le i,j,k \le d$ satisfying 
 $A_i A_j = A_j A_i = \sum_{k=0}^d p_{ij}^k A_k$.
 \end{itemize}
 We call $\BMA$ the \emph{Bose-Mesner algebra} of the association scheme $(X,\cR)$.
 
Up to a choice of ordering of relations and ordering of vertices, the correspondence between 
Bose-Mesner algebras and association schemes is immediate and we often
work with the adjacency matrix $A_i$ in place of the graph $(X,R_i)$.

The vector space $\BMA=\spn_\re\{A_0,\ldots,A_d\}$ has a second basis of \emph{primitive (matrix) idempotents} $\{E_0,\ldots,E_d\}$: $\sum_{j=0}^d E_j=I$, $E_i E_j = \delta_{ij} E_i$, $E_i=E_i^\top$. 
The eigenvalues $P_{ji}$ and the dual eigenvalues $Q_{ji}$ of the association scheme satisfy
\begin{equation}
\label{Eqn:PQ} 
A_i = \sum_{j=0}^d P_{ji} E_j
\quad \text{and}\quad
E_i = \frac{1}{|X|} \sum_{j=0}^d Q_{ji} A_j.
\end{equation}
Since $\BMA$ is closed under entrywise multiplication, there exist \emph{Krein parameters} $q_{ij}^k$, $0\le i,j,k\le d$ satisfying $E_i \circ E_j = \frac{1}{|X|} \sum_{k=0}^d q_{ij}^k E_k$. (See \cite[Section 2.2-2.3]{bcn}.)

\subsection{Scaffolds}
\label{Subsec:scaffolds}

Let $(X,\cR)$ be an association scheme with Bose-Mesner algebra $\BMA$. These matrices act, in the obvious way, on the \emph{standard 
module} $V=\cx^X$ of all complex-valued functions on $X$ with standard basis of column vectors
$\{ \hat{x} \mid x\in X\}$.  This space is equipped with the corresponding positive definite Hermitian inner product 
$\langle v,w\rangle = v^\dagger w$ (where $\cdot^\dagger$ denotes conjugate transpose) 
satisfying $\langle \hat{x},\hat{y} \rangle = \delta_{x,y}$ for $x,y\in X$. We identify $V$ with its dual space 
$V^\dagger$ of linear functionals and  view matrices in $\Mat_X(\cx)$ as second order tensors. More generally, we will presently define
a scaffold with $m$ roots (or $m^{\rm th}$ order scaffold) as a certain type of tensor belonging  to 
$$  V^{\otimes m} = \underbrace{V \otimes V \otimes \cdots \otimes V}_{m} $$
with standard basis consisting of simple tensors of the form 
$\hat{x}_1 \otimes \hat{x}_2 \otimes \cdots \otimes \hat{x}_m$ where
$x_1,x_2,\ldots, x_m \in X$.

For a  graph $F=(V(F),E(F))$, an ordered set $R = \{r_1,\ldots,r_m\}$ of nodes in $F$ called \emph{roots}, and a function 
$w: E(F) \rightarrow \Mat_X(\cx)$, we define
\begin{equation}
\label{Eq:scaffshort}
\sS(F,R;w ) = \sum_{ \varphi:V(F) \rightarrow X} \quad \left( \prod_{\substack{ e\in E(F)  \\ e=\{a,b\} }} w(e)_{\varphi(a),\varphi(b)} \right) \widehat{\varphi(r_1)} \otimes \widehat{\varphi(r_2)} \otimes \cdots  \otimes \widehat{\varphi(r_m)}.
\end{equation}
We call $F$ the ``diagram'' of the scaffold $\sS(F,R;w)$, use red solid nodes  to depict the roots, and label each edge $e$ with the matrix $w(e)$.
 We identify the scaffold $\sS(F,R;w)$ with this pictorial representation of its data, being careful
to consistently order the roots by spacial placement when two scaffolds appear in the same equation. For instance, the  matrix
$A=[A_{xy}] \in \Mat_X(\cx)$, viewed as the second-order tensor $\sum_{x,y \in X} A_{xy} \ \hat{x} \otimes \hat{y}$, is denoted by  \inlineSDMatrix{M}{$A$}{1.5}. In this paper, all examples are symmetric matrices. So we will omit the directions on edges of the diagrams\footnote{In fact, our definition of a scaffold here is specialized to undirected graphs.}.

When studying an association scheme $(X,\cR)$ with adjacency matrices $A_i$, two families of third order scaffolds of fundamental importance \cite{terPQ,jaeger} are 
\begin{equation}
\label{Eqn:DeltaandWye}
 \inlineSDelta{D}{$A_i$,$A_j$,$A_k$}{1} = \hspace{-0.1in} \sum_{\substack{x,y,z\in X \\ x\Rrel{i}z, x\Rrel{j}y, y\Rrel{k}z}} 
 \hat{x} \otimes \hat{y} \otimes \hat{z} ~, \quad \inlineSWye{Y}{$A_i$,$A_j$,$A_k$}{1} =  \hspace{-0.1in} \sum_{\substack{x,y,z,u\in X \\ x\Rrel{i}u,y\Rrel{j}u,z\Rrel{k}u}}  
 \hat{x} \otimes \hat{y} \otimes \hat{z} ~ .
 \end{equation}
We next consider the vector space spanned
by all scaffolds with a given diagram  \cite[Section~3.2]{WJMscaff}. In particular, define
\begin{itemize}
\item  
\SWDelta{$\BMA$} = $\spn \left\{  \inlineSDelta{D}{$L$,$M$,$N$}{1} \middle| L,M,N \in \BMA \right\}$,
\item \SWWye{$\BMA$} = $\spn \left\{  \inlineSWye{D}{$L$,$M$,$N$}{1} \middle| L,M,N \in \BMA \right\}$.
\end{itemize}

As explained in our introduction, scaffolds help us count homomorphisms. 

\begin{lemma}
\label{Lem:homsarescaffolds}
Let  $F$ be a graph, let $G$ be a graph on vertex set $X$ with adjacency matrix $A$. Define
$w(e)=A$ for all edges $e$ of $F$. Then $\sS(F,\emptyset;w)= \mathsf{hom}(F,G)$, the number of 
graph homomorphisms from $F$ to $G$. \qed
\end{lemma}

We will perform local operations on scaffolds that preserve their value, using Proposition 1.5 in \cite{WJMscaff}. These include loop removal, removal of a non-root vertex of degree one, and series and parallel reduction. One may use the definition to directly verify the following
identities:
\begin{itemize} \label{List:seriespar}
    \item[{\sf sr}] If $M$ has constant row sum $\alpha$, then \inlineSpendant{$\scriptstyle{M}$}{1.0} $=\alpha$ \inlineRedDot
    \item[{\sf sr'}] If $M$ has constant diagonal $\alpha$, then \inlineSloop{$\ \scriptstyle{M}$} $=\alpha$ \inlineRedDot
    \item[{\sf sr1}] For $M,N\in \Mat_X(\cx)$, 
\begin{tikzpicture}[baseline=(A1),black,node distance=0.5cm,
solidvert/.style={draw, circle,  fill=red, inner sep=1.8pt},
hollowvert/.style={  draw,  circle,  fill=white,  inner sep=1.8pt}]
\node[solidvert] (A1) at (0,0) {};
\node[hollowvert] (A2) at (0.75,0) {};
\node[solidvert] (A3) at (1.5,0) {};
   \path[->] (A1)   [thick] edge node  [above] {$\scriptstyle{M}$} (A2);
   \path[->] (A2)   [thick] edge node  [above] {$\scriptstyle{N}$} (A3);
\node (B) at (1.9,0) {$=$};
\node[solidvert] (C1) at (2.3,0) {};
\node[solidvert] (C2) at (3.2,0) {};
   \path[->] (C1)   [thick] edge node  [above] {$\scriptstyle{MN}$} (C2);
\end{tikzpicture}
    \item[{\sf sr1'}] For $M,N\in \Mat_X(\cx)$, 
\begin{tikzpicture}[baseline=(A1),black,node distance=0.5cm,
solidvert/.style={draw, circle,  fill=red, inner sep=1.8pt},
hollowvert/.style={  draw,  circle,  fill=white,  inner sep=1.8pt}]
\node[solidvert] (A1) at (0,0) {};
\node[solidvert] (A2) at (1,0) {};
   \path[->] (A1)   [thick,bend left=45] edge node  [above] {$\scriptstyle{M}$} (A2);
   \path[->] (A1)   [thick, bend right] edge node  [above] {$\scriptstyle{N}$} (A2);
\node (B) at (1.4,0) {$=$};
\node[solidvert] (C1) at (1.8,0) {};
\node[solidvert] (C2) at (2.8,0) {};
   \path[->] (C1)   [thick] edge node  [above] {$\scriptstyle{M\circ N}$} (C2);
\end{tikzpicture}
\end{itemize}

\subsection{Extra regularity}
\label{Subsec:tr}

An association scheme $(X,\cR)$ with Bose-Mesner algebra $\BMA$ is \emph{triply regular} if, for all $x,y,z\in X$ and all 
$0\le i,j,k \le d$, $\upsilon(x,y,z):=\left| \left\{ \ u\in X \ : \  x\Rrel{i}u, y\Rrel{j} u, z\Rrel{k} u \ \right\}\right|$ depends only on 
$i,j,k$ and the three relations joining $x,y,z$ and not on the choice of $x,y,z$ themselves.  Jaeger proved  that $(X,\cR)$ is triply regular if and only if \SWWye{$\BMA$} $\subseteq$ \SWDelta{$\BMA$}   \cite[Proposition 7(ii)]{jaeger}. 
If we use $\upsilon_{r,s,t}^{i,j,k}$ to denote $\upsilon(x,y,z)$ 
when $x \Rrel{r} z$, $x \Rrel{s} y$ and $y \Rrel{t} z$, then
the scaffold equation
$$ \inlineSWye{W}{$A_i$,$A_j$,$A_k$}{1} = \sum_{r,s,t} \upsilon^{i,j,k}_{r,s,t}  \inlineSDelta{D}{$A_r$,$A_s$,$A_t$}{1}$$
holds for all $i, j, k$.

The association scheme $(X,\cR)$ is \emph{dually triply regular} if  \ \ \SWDelta{$\BMA$} \ $\subseteq$ \\ 
\SWWye{$\BMA$}  \cite[Proposition 8(ii)]{jaeger}
and \emph{exactly triply regular} if it is both triply regular and dually triply regular. 

\begin{theorem}[Terwilliger \cite{terwnotes2009} (see {\cite[Theorem~3.8]{WJMscaff}})] 
\label{Thm:Terw}
Let $(X,\cR)$ be a symmetric association scheme with minimal Schur idempotents $A_0,\ldots,A_d$, primitive (matrix) idempotents $E_0,\ldots,E_d$,
intersection numbers $p_{ij}^k$ and Krein parameters $q_{ij}^k$ ($0\le i,j,k\le d$). The set
$\left\{ \inlineSDelta{D}{$A_i$,$A_j$,$A_k$}{0.8} \middle| \ p_{ij}^k > 0 \ \right\}$ is an orthogonal basis for  \!\!\!\!\! \SWDelta{$\BMA$} and \\
$\left\{ \, \inlineSWye{D}{$E_i$,$E_j$,$E_k$}{0.8} \  \middle| \ q_{ij}^k > 0 \ \right\}$ is an orthogonal basis for  \SWWye{$\BMA$}.  \qed
\end{theorem}

Let us denote by $N_p$ the number of ordered triples $(i,j,k)$ with $p_{ij}^k>0$ and by $N_q$ the number of ordered triples $(i,j,k)$ with $q_{ij}^k>0$.
The following lemma follows from Jaeger's propositions.

\begin{lemma}
\label{Lem:exactlytrip}
    If $(X,\cR)$ is an exactly triply regular association scheme with Bose-Mesner algebra $\BMA$, then \SWDelta{$\BMA$} $=$ \SWWye{$\BMA$}.
    If $N_p=N_q$ and $(X,\cR)$ is either triply regular or dually triply regular, then $(X,\cR)$ is exactly triply regular. \qed
\end{lemma}

\begin{definition}
\label{Def:D-Ypar}
    Let $(X,\cR)$ be an exactly triply regular $d$-class association scheme with Bose-Mesner algebra having an ordered basis 
    $A_0,\ldots,A_d$ of adjacency matrices and an ordered basis $E_0,\ldots,E_d$ of primitive idempotents. The 
    \emph{Delta-Wye parameters} of $(X,\cR)$  are those $\left\{ \sigma^{i,j,k}_{r,s,t} \middle| p_{ij}^k>0, \ q_{rs}^t >0 \right\}$ and $\left\{ \tau_{i,j,k}^{r,s,t} \middle| p_{ij}^k>0, \ q_{rs}^t >0 \right\}$
    satisfying the equations
    \begin{equation}
    \label{Eqn:sigma}
    \inlineSDelta{D}{$\scriptstyle{A_i}$, $\scriptstyle{A_j}$, $\scriptstyle{A_k}$}{0.8} = \sum_{q_{rs}^t > 0} \sigma^{i,j,k}_{r,s,t} \inlineSWye{Y}{$\scriptstyle{E_r}$,$\scriptstyle{E_s}$,$\scriptstyle{E_t}$}{0.8}
    \end{equation}
    and
        \begin{equation}
    \label{Eqn:tau}
        \inlineSWye{Y}{$\scriptstyle{E_r}$,$\scriptstyle{E_s}$,$\scriptstyle{E_t}$}{0.8} = \sum_{p_{ij}^k > 0} \tau_{i,j,k}^{r,s,t}
        \inlineSDelta{D}{$\scriptstyle{A_i}$,$\scriptstyle{A_j}$,$\scriptstyle{A_k}$}{0.8} .
            \end{equation}
\end{definition}

\begin{remark}
    The two sets of coefficients are mutual inverses. So if one knows all $\sigma^{i,j,k}_{r,s,t}$, one may derive from these the parameters 
    $\tau_{i,j,k}^{r,s,t}$ and conversely.
\end{remark}

Applying (\ref{Eqn:PQ}) to Definition~\ref{Def:D-Ypar} gives the following equations
\begin{equation}
\label{Eqn:D-Y}
 \inlineSDelta{D}{$\scriptstyle{A_i}$, $\scriptstyle{A_j}$, $\scriptstyle{A_k}$}{0.8} = 
 \sum_{q_{ab}^c > 0} \sigma^{i,j,k}_{a,b,c} \inlineSWye{Y}{$\scriptstyle{E_a}$,$\scriptstyle{E_b}$,$\scriptstyle{E_c}$}{0.8} = 
 \frac{1}{|X|^3} \sum_{q_{ab}^c > 0} \sigma^{i,j,k}_{a, b, c }   \sum_{r,s,t} Q_{ra} Q_{sb}Q_{tc}  \inlineSWye{Y}{$\scriptstyle{A_r}$,$\scriptstyle{A_s}$,$\scriptstyle{A_t}$}{0.8}
\end{equation}
and
\begin{equation}
\label{Eqn:Y-D}
\inlineSWye{Y}{$\scriptstyle{A_i}$,$\scriptstyle{A_j}$,$\scriptstyle{A_k}$}{0.8}=
\sum_{a,b,c} P_{ai}P_{bj}P_{ck} \inlineSWye{Y}{$\scriptstyle{E_a}$,$\scriptstyle{E_b}$,$\scriptstyle{E_c}$}{0.8} = 
 \sum_{a,b,c} P_{ai}P_{bj}P_{ck} \sum_{p_{rs}^t > 0} \tau_{r,s,t}^{a,b,c} \inlineSDelta{D}{$\scriptstyle{A_r}$,$\scriptstyle{A_s}$,$\scriptstyle{A_t}$}{0.8} 
 \end{equation}
 which we will use in the proof of Theorem~\ref{Thm:EqualScaffolds}. Note that, while the expansion in (\ref{Eqn:Y-D}) is unique by Theorem \ref{Thm:Terw}, there are many solutions to the equation in (\ref{Eqn:D-Y}); it is important that we consistently use this one in our proof.

\begin{definition}
\label{Def:sameDeltaWye}
    Let $(X,\cR)$ and $(Y,\cS)$ be exactly triply regular $d$-class association schemes. We say $(X,\cR)$ and $(Y,\cS)$ \emph{have the same Delta-Wye parameters} if there exist orderings $A_0,A_1,\ldots,A_d$ and $A'_0,A'_1,\ldots,A'_d$ {of their respective adjacency matrices}, and there exist orderings $E_0,E_1,\ldots,E_d$ and $E'_0,E'_1,\ldots,E'_d$  {of  their respective primitive idempotents} such that every Delta-Wye parameter $\sigma_{r,s,t}^{i,j,k}$ for $(X,\cR)$ is equal to the corresponding 
    Delta-Wye parameter for $(Y,\cS)$.
\end{definition}

\begin{definition}
\label{Def:corresponding}
Let $(X,\cR)$  and $(Y,\cS)$  be  exactly triply regular $d$-class association schemes with the same Delta-Wye parameters with respect to  orderings $A_0,A_1,\ldots,A_d$ and $A'_0,A'_1,\ldots$, $A'_d$ of  {their respective adjacency matrices}, and orderings $E_0,E_1,\ldots,E_d$ and $E'_0,E'_1,\ldots,E'_d$ of their respective primitive idempotents. The bijection $A_i\mapsto A_i'$ extends
 linearly to a vector space isomorphism $\kappa: \BMA \rightarrow \BMA'$
carrying a matrix $M=\sum_{j=0}^d c_jA_j$ to the matrix $\kappa(M)=\sum_{j=0}^d c_jA'_j$ which we denote $M'$. (Note that $\kappa(E_j)=E_j'$ by Proposition \ref{Prop:same}.)  We call $M'$ the matrix in the Bose-Mesner algebra of $(Y,\cS)$ \emph{corresponding to} $M$. In
the special case where $M$ is the adjacency matrix of a graph $G$ on vertex set $X$, the matrix $M'$ is the adjacency matrix of some graph $G'$ on vertex set $Y$; we call $G'$ the graph in $(Y,\cS)$ \emph{corresponding to} $G$.
\end{definition}

We do not know whether two exactly triply regular association schemes with the same intersection numbers must have the same Delta-Wye parameters. In Section \ref{Sec:problems}, we ask if these parameters are functions of the $p_{ij}^k$'s.

\begin{proposition}
\label{Prop:same}
    If $(X,\cR)$ and $(Y,\cS)$ are exactly triply regular association schemes having the same Delta-Wye parameters, then they also have the same eigenvalues, dual eigenvalues, intersection numbers and Krein parameters under the appropriate consistent orderings of relations and primitive idempotents. The linear map $\kappa$ in Definition \ref{Def:corresponding} is a Bose-Mesner isomorphism: $\kappa(MN)=\kappa(M)\kappa(N)$ and $\kappa(M\circ N)=\kappa(M)\circ \kappa(N)$.
\end{proposition}

\begin{proof}
   The reader may sum coefficients (convert root nodes to hollow nodes preserving equality) in Equation (\ref{Eqn:sigma}) to find $p_{kk}^0 = \sigma_{0,0,0}^{k,k,0}$ and $p_{ij}^k p_{kk}^0 = \sigma_{0,0,0}^{i,j,k}$ and the remaining association scheme parameters are recoverable from the intersection numbers.
   
   Since $\kappa(A_i)=A'_i$ is a 01-matrix and both $A_i \circ A_j=\delta_{ij} A_i$ and $A'_i \circ A'_j = \delta_{ij} A'_i$ hold for this pair of bases, we have $\kappa(M\circ N)=\kappa(M)\circ \kappa(N)$ by linearity. Finally, 
   $$ \kappa( A_i A_j ) = \kappa\left( \sum_{k=0}^d p_{ij}^k A_k \right) = \sum_{k=0}^d p_{ij}^k A'_k = A'_i A'_j = \kappa(A_i)\kappa(A_j). \qed $$
\end{proof}
\section{Homomorphism counting}
\label{Sec:homo}
Given a planar graph $F$ and a graph $G$, let $w_G$ be the weight function mapping each edge of $F$ to the adjacency matrix $A_G$.
We compute the scaffold $\sS(F, \emptyset; w_G)$ to count the number of graph homomorphisms from $F$ to $G$ using Epifanov's Theorem on plane graphs.

\subsection{Epifanov's Theorem}
\label{Subsec:Epifanov}

A \emph{plane graph} \cite[p83]{diestel} is an embedding of a planar graph and we do not differentiate between embeddings equivalent under ambient isotopy.
We allow the following local operations on plane graphs. Each modifies an embedded graph only within a closed disk with the understanding that this disk contains no part of the embedding other than what is shown.

\begin{center}
\def\arraystretch{1.5}
\begin{tabular}{lcll}
\textsf{loop}:  &
\begin{tikzpicture}[baseline=(A1),black,node distance=0.5cm,  
solidvert/.style={draw, circle,  fill=red, inner sep=1.8pt},
hollowvert/.style={  draw,  circle,  fill=white,  inner sep=1.8pt},
  every loop/.style={min distance=40pt,in=-30,out=90,looseness=20}]
\node[hollowvert] (A1) at (0,0) {};
\node (goesto)  at (1.5,0) {$\rightsquigarrow$};
\node[hollowvert] (B1) at (2.5,0) {};
\path (A1) edge[loop,in=45,out=-45,looseness=20] node {}  (A1);
\draw (A1) -- (-0.2,0.15); \draw (A1) -- (-0.3,0);  \draw (A1) -- (-0.2,-0.15);
\draw (B1) -- (2.3,0.15); \draw (B1) -- (2.2,0);  \draw (B1) -- (2.3,-0.15);
\end{tikzpicture}& deletion of a loop \\
\textsf{pendent}:   &
\begin{tikzpicture}[baseline=(A1),black,node distance=0.5cm,  
solidvert/.style={draw, circle,  fill=red, inner sep=1.8pt},
hollowvert/.style={  draw,  circle,  fill=white,  inner sep=1.8pt},
  every loop/.style={min distance=40pt,in=-30,out=90,looseness=20}]
\node[hollowvert] (A1) at (0,0) {}; 
\node[hollowvert] (A2) at (1,0) {}; 
\node (goesto)  at (2,0) {$\rightsquigarrow$};
\node[hollowvert] (B1) at (3,0) {};
\path (A1)  edge node {}  (A2);
\draw (A1) -- (-0.2,0.15); \draw (A1) -- (-0.3,0);  \draw (A1) -- (-0.2,-0.15);
\draw (B1) -- (2.8,0.15); \draw (B1) -- (2.7,0);  \draw (B1) -- (2.8,-0.15);
\end{tikzpicture} & deletion of a vertex of degree one and pendent edge\\
\textsf{series}:  & 
\begin{tikzpicture}[baseline=(A1),black,node distance=0.5cm,  
solidvert/.style={draw, circle,  fill=red, inner sep=1.8pt},
hollowvert/.style={  draw,  circle,  fill=white,  inner sep=1.8pt},
  every loop/.style={min distance=40pt,in=-30,out=90,looseness=20}]
\node[hollowvert] (A1) at (-1,0) {};
\node[hollowvert] (A2) at (0,0) {};
\node[hollowvert] (A3) at (1,0) {};
\node (goesto)  at (2,0) {$\rightsquigarrow$};
\node[hollowvert] (B1) at (3,0) {};
\node[hollowvert] (B2) at (4,0) {};
    \path[-] (A1) edge node {} (A2);     \path[-] (A2) edge node {} (A3);
        \path[-] (B1) edge node {} (B2);
\draw (A1) -- (-1.2,0.15); \draw (A1) -- (-1.3,0);  \draw (A1) -- (-1.2,-0.15);
\draw (A3) -- (1.2,0.15); \draw (A3) -- (1.3,0);  \draw (A3) -- (1.2,-0.15);
\draw (B1) -- (2.8,0.15); \draw (B1) -- (2.7,0);  \draw (B1) -- (2.8,-0.15);
\draw (B2) -- (4.2,0.15); \draw (B2) -- (4.3,0);  \draw (B2) -- (4.2,-0.15);
\end{tikzpicture}& a series reduction\\
\textsf{parallel}:   &
\begin{tikzpicture}[baseline=(A1),black,node distance=0.5cm,  
solidvert/.style={draw, circle,  fill=red, inner sep=1.8pt},
hollowvert/.style={  draw,  circle,  fill=white,  inner sep=1.8pt},
  every loop/.style={min distance=40pt,in=-30,out=90,looseness=20}]
\node[hollowvert] (A1) at (0,0) {};
\node[hollowvert] (A2) at (1,0) {};
\node (goesto)  at (2,0) {$\rightsquigarrow$};
\node[hollowvert] (B1) at (3,0) {};
\node[hollowvert] (B2) at (4,0) {};
    \path[-,bend left] (A1) edge node {} (A2);    \path[-,bend right] (A1) edge node {} (A2);     
        \path[-] (B1) edge node {} (B2);
\draw (A1) -- (-0.2,0.15); \draw (A1) -- (-0.3,0);  \draw (A1) -- (-0.2,-0.15);
\draw (A2) -- (1.2,0.15); \draw (A2) -- (1.3,0);  \draw (A2) -- (1.2,-0.15);
\draw (B1) -- (2.8,0.15); \draw (B1) -- (2.7,0);  \draw (B1) -- (2.8,-0.15);
\draw (B2) -- (4.2,0.15); \draw (B2) -- (4.3,0);  \draw (B2) -- (4.2,-0.15);
\end{tikzpicture} & a parallel reduction\\
\textsf{Delta}:&  
\begin{tikzpicture}[baseline=(A1),black,node distance=0.5cm,  
solidvert/.style={draw, circle,  fill=red, inner sep=1.8pt},
hollowvert/.style={  draw,  circle,  fill=white,  inner sep=1.8pt},
  every loop/.style={min distance=40pt,in=-30,out=90,looseness=20}]
\def\sc {0.3}
\node[hollowvert] (A1) at (-1*\sc,0*\sc) {};
\node[hollowvert] (A2) at (2*\sc,0*\sc) {};
\node[hollowvert] (A3) at (0.5*\sc,1.5*1.732*\sc) {};
\def\xsh {2.5}
\node[hollowvert] (B1) at (\xsh-1*\sc,0*\sc) {};
\node[hollowvert] (B2) at (\xsh+2*\sc,0*\sc) {};
\node[hollowvert] (B3) at (\xsh+0.5*\sc,1.5*1.732*\sc) {};
\node[hollowvert] (B4) at (\xsh+0.5*\sc,1.5*.732*\sc) {};
  \path[-] (A1)  edge node  [] {} (A2);
  \path[-] (A2)  edge node  [] {} (A3);
  \path[-] (A3)  edge node  [] {} (A1);
  \path[-] (B1)  edge node  [] {} (B4);
  \path[-] (B2)  edge node  [] {} (B4);
  \path[-] (B3)  edge node  [] {} (B4);
\node (goesto) at (\xsh-1.2,0.75*\sc) {$\rightsquigarrow$};
\draw (A1) --  (-1.9*\sc,0*\sc); \draw (A1) --  (-1.6*\sc,-0.5*\sc); \draw (A1) --  (-1.3*\sc,-0.8*\sc); 
\draw (A2) --  (2.9*\sc,0*\sc); \draw (A2) --  (2.6*\sc,-0.5*\sc); \draw (A2) --  (2.3*\sc,-0.8*\sc); 
\draw (A3) --  (-0.2*\sc,3.4*\sc); \draw (A3) --  (0.5*\sc,3.5*\sc); \draw (A3) --  (1.2*\sc,3.4*\sc); 
\draw (B1) --  (\xsh-1.9*\sc,0*\sc); \draw (B1) --  (\xsh-1.6*\sc,-0.5*\sc); \draw (B1) --  (\xsh-1.3*\sc,-0.8*\sc); 
\draw (B2) --  (\xsh+2.9*\sc,0*\sc); \draw (B2) --  (\xsh+2.6*\sc,-0.5*\sc); \draw (B2) --  (\xsh+2.3*\sc,-0.8*\sc); 
\draw (B3) --  (\xsh-0.2*\sc,3.4*\sc); \draw (B3) --  (\xsh+0.5*\sc,3.5*\sc); \draw (B3) --  (\xsh+1.2*\sc,3.4*\sc); 
\end{tikzpicture} & a Delta-Wye transformation\\
\textsf{Wye}: & 
\begin{tikzpicture}[baseline=(A1),black,node distance=0.5cm,  
solidvert/.style={draw, circle,  fill=red, inner sep=1.8pt},
hollowvert/.style={  draw,  circle,  fill=white,  inner sep=1.8pt},
  every loop/.style={min distance=40pt,in=-30,out=90,looseness=20}]
\def\sc {0.3}
\node[hollowvert] (A1) at (-1*\sc,0*\sc) {};
\node[hollowvert] (A2) at (2*\sc,0*\sc) {};
\node[hollowvert] (A3) at (0.5*\sc,1.5*1.732*\sc) {};
\node[hollowvert] (A4) at (0.5*\sc,1.5*.732*\sc) {};
\def\xsh {2.5}
\node[hollowvert] (B1) at (\xsh-1*\sc,0*\sc) {};
\node[hollowvert] (B2) at (\xsh+2*\sc,0*\sc) {};
\node[hollowvert] (B3) at (\xsh+0.5*\sc,1.5*1.732*\sc) {};
  \path[-] (A1)  edge node  [] {} (A4);
  \path[-] (A2)  edge node  [] {} (A4);
  \path[-] (A3)  edge node  [] {} (A4);
  \path[-] (B1)  edge node  [] {} (B2);
  \path[-] (B2)  edge node  [] {} (B3);
  \path[-] (B3)  edge node  [] {} (B1);
\node (goesto) at (\xsh-1.2,0.75*\sc) {$\rightsquigarrow$};
\draw (A1) --  (-1.9*\sc,0*\sc); \draw (A1) --  (-1.6*\sc,-0.5*\sc); \draw (A1) --  (-1.3*\sc,-0.8*\sc); 
\draw (A2) --  (2.9*\sc,0*\sc); \draw (A2) --  (2.6*\sc,-0.5*\sc); \draw (A2) --  (2.3*\sc,-0.8*\sc); 
\draw (A3) --  (-0.2*\sc,3.4*\sc); \draw (A3) --  (0.5*\sc,3.5*\sc); \draw (A3) --  (1.2*\sc,3.4*\sc); 
\draw (B1) --  (\xsh-1.9*\sc,0*\sc); \draw (B1) --  (\xsh-1.6*\sc,-0.5*\sc); \draw (B1) --  (\xsh-1.3*\sc,-0.8*\sc); 
\draw (B2) --  (\xsh+2.9*\sc,0*\sc); \draw (B2) --  (\xsh+2.6*\sc,-0.5*\sc); \draw (B2) --  (\xsh+2.3*\sc,-0.8*\sc); 
\draw (B3) --  (\xsh-0.2*\sc,3.4*\sc); \draw (B3) --  (\xsh+0.5*\sc,3.5*\sc); \draw (B3) --  (\xsh+1.2*\sc,3.4*\sc); 
\end{tikzpicture} & a Wye-Delta transformation\\
\end{tabular}
\end{center}

\begin{theorem}[Epifanov (see {\cite[Proposition 5]{jaeger}})]
\label{Thm:Epifanov}
Let $F$ be any connected plane graph. Then there exists a sequence of plane graphs $F_0,F_1,\ldots,F_{\ell}$ with the following properties.
\begin{itemize}
\item[(i)] $F_0=F$ and $F_{\ell}$ is a graph with one vertex and no edges
\item[(ii)] up to ambient isotopy, $F_{h+1}$ is obtained from $F_h$ by just one of the above local transformations (\textsf{loop}, \textsf{pendent}, \textsf{series}, \textsf{parallel}, \textsf{Delta}, or \textsf{Wye}), for $0\le h< \ell$.
\qed
\end{itemize}
\end{theorem}

\subsection{A technical theorem and the main result}
\label{Subsec:main}

In this section, we assume  $(X,\cR)$ and $(Y, \cS)$ are exactly triply regular association schemes having the same Delta-Wye parameters with respect to  orderings $A_0,\ldots,A_d$ and $A'_0,\ldots,A'_d$ of  {their respective adjacency matrices}, and orderings $E_0,E_1,\ldots,E_d$ and $E'_0,E'_1,\ldots,E'_d$ of their respective primitive idempotents.
We use $\BMA$ and $\BMA'$ to denote the Bose-Mesner algebras of $(X,\cR)$ and $(Y, \cS)$, respectively. To 
avoid confusion, we denote the composition of functions using the symbol $\bullet$.

\begin{theorem}
\label{Thm:EqualScaffolds}
Let $(X,\cR)$ and $(Y, \cS)$ be exactly triply regular association schemes having the same Delta-Wye parameters
and let $\kappa: \BMA \rightarrow \BMA'$ be as given in 
Proposition \ref{Prop:same}.
Let $F$ be a connected plane graph (possibly with loops and multiple edges) and consider any edge weights $w: E(F) \rightarrow \BMA$. Then $\sS(F, \emptyset; w)=\sS(F,\emptyset; \kappa \bullet w)$. 
\end{theorem}
\begin{proof}
Write $w'=\kappa \bullet w$ so that $w':E(F) \rightarrow \BMA'$.
First choose an embedding and view $F$ as a plane graph.  
Let $F_0, F_1, \ldots, F_{\ell}$ be a sequence of plane graphs satisfying the conditions in Theorem~\ref{Thm:Epifanov}.

 Since $\BMA=\spn\{A_0, A_1, \ldots, A_d\}$ and $\BMA'=\spn\{A'_0, A'_1, \ldots, A'_d\}$,
 it is sufficient to show inductively that, for $h=0,1,\ldots,\ell$, there exist $m_h$, $\alpha_{h,m}$ and weight functions  
 $$w_{h,m}: E(F_h) \rightarrow \{A_0, A_1, \ldots, A_d\} $$
 such that
 \begin{equation}
 \label{Eqn:Induction}
\sS(F, \emptyset; w) = \sum_{m=1}^{m_h} \alpha_{h,m} \ \sS(F_h, \emptyset; w_{h,m})
\quad  \text{and}\quad
 \sS(F, \emptyset; w') = \sum_{m=1}^{m_h} \alpha_{h,m} \ \sS(F_h, \emptyset; \kappa \bullet w_{h,m}) ~ .
 \end{equation}
 Note that, since $F_\ell$ has no edges, $ \sS(F_\ell, \emptyset; \omega)=|X|$ for $\omega: \emptyset \rightarrow \Mat_X(\cx)$ and  the result follows: 
 $$ \sS(F, \emptyset; w) = |X| \sum_{m=1}^{m_\ell} \alpha_{\ell,m} = |Y| \sum_{m=1}^{m_\ell} \alpha_{\ell,m} =  \sS(F, \emptyset; w') . $$

 When $h=0$, Equation~\ref{Eqn:Induction} follows because $w(e) \in \BMA$, for each $e\in E(F)$, and $w'=\kappa\bullet w$.
 For $h>0$, we show Equation~\ref{Eqn:Induction} holds  for appropriate $\alpha_{h,m}$ and $w_{h,m}$ by considering each type of local transformation occurring in Theorem~\ref{Thm:Epifanov} and its effect on the associated scaffolds.  We include the edge weights in the following scaffolds to highlight the change of the weight functions due to the local transformations.

\medskip
\noindent
\textsf{loop}: $F_{h+1}$ is obtained from $F_h$ by deleting a loop $e$.  Set $m_{h+1}=m_h$,
consider $1\le m\le m_h$, and suppose $w_{h,m}(e)=A_r$. Define $w_{h+1,m}: E(F_{h+1}) \rightarrow \BMA$ to be the restriction 
of $w_{h,m}$ to $E(F_h)\backslash \{e\}$ so that 
\begin{equation}
\label{Eqn:loop}
\begin{tikzpicture}[baseline=(A1),black,node distance=0.5cm,  
solidvert/.style={draw, circle,  fill=red, inner sep=1.8pt},
hollowvert/.style={  draw,  circle,  fill=white,  inner sep=1.8pt},
  every loop/.style={min distance=40pt,in=-30,out=90,looseness=20}]


%
\node[hollowvert] (A1) at (0,0) {};
\path (A1) edge[loop,in=45,out=-45,looseness=20, right] node [midway] {\scriptsize $A_r$}  (A1);
\draw (A1) -- (-0.2,0.15); \draw (A1) -- (-0.3,0);  \draw (A1) -- (-0.2,-0.15);
\draw (-1,0) node[anchor=west]{$\sS( \hspace{1.75cm}, \emptyset; w_{h,m}) =$};
\draw (3.2,0) node[anchor=west]{$\delta_{0,r}\  \sS( \hspace{0.75cm}, \emptyset; w_{h+1,m}),$};
\node[hollowvert] (B1) at (4.95,0) {};
\draw (B1) -- (4.75,0.15); \draw (B1) -- (4.75,0);  \draw (B1) -- (4.75,-0.15);

%

%

\end{tikzpicture}
\end{equation}
by Rule \textsf{sr'} in Section \ref{Subsec:scaffolds} above
and the same equation holds after replacing $w_{h,m}$ and $w_{h+1,m}$ with $\kappa \bullet w_{h,m}$ and $\kappa \bullet w_{h+1,m}$ respectively. Summing over $m$ with coefficients $\alpha_{h+1,m} = \delta_{0,r} \alpha_{h,m}$ yields Equations (\ref{Eqn:Induction}) with $h$ replaced by $h+1$.

\medskip
\noindent
\textsf{pendent}: 
$F_{h+1}$ is obtained from $F_h$ by deletion of a degree one vertex and the sole incident edge $e$. Let $m_{h+1}=m_h$ and, for each $1\le m\le m_h$, 
setting $A_r=w_{h,m}(e)$, define $w_{h+1,m}: E(F_{h+1}) \rightarrow \BMA$ 
to be the restriction  of $w_{h,m}$ to $E(F_h)\backslash \{e\}$. Then we have 
\begin{equation}
\label{Eqn:leaf}
\begin{tikzpicture}[baseline=(A1),black,node distance=0.5cm,  
solidvert/.style={draw, circle,  fill=red, inner sep=1.8pt},
hollowvert/.style={  draw,  circle,  fill=white,  inner sep=1.8pt},
  every loop/.style={min distance=40pt,in=-30,out=90,looseness=20}]
\node[hollowvert] (A1) at (0,0) {}; 
\node[hollowvert] (A2) at (1,0) {}; 
\path (A1)  edge [above] node [] {\scriptsize $A_r$}  (A2);
\draw (A1) -- (-0.2,0.15); \draw (A1) -- (-0.3,0);  \draw (A1) -- (-0.2,-0.15);
\draw (-1,0) node[anchor=west]{$\sS( \hspace{1.75cm}, \emptyset; w_{h,m}) =$};
\def\xsh {0.2}
\node[hollowvert] (B1) at (5+\xsh,0) {};
\draw (3+\xsh,0) node[anchor=west]{$p_{r,r}^0 \  \sS( \hspace{1cm}, \emptyset; w_{h+1,m})$};
\draw (B1) -- (4.8+\xsh,0.15); \draw (B1) -- (4.7+\xsh,0);  \draw (B1) -- (4.8+\xsh,-0.15);
\end{tikzpicture}
\end{equation}
by Rule \textsf{sr} in Section \ref{Subsec:scaffolds} above. Since $A'_r=\kappa(A_r)$ has the same row sum as $A_r$, the same equation holds when $w_{h,m}$ and $w_{h+1,m}$ are replaced by $\kappa \bullet w_{h,m}$ and $\kappa \bullet w_{h+1,m}$, respectively. Choosing coefficients
$\alpha_{h+1,m} = p_{r,r}^0 \alpha_{h,m}$ for $1\le m\le m_h$ and summing gives the induction step in this case.

\medskip
\noindent
\textsf{series}: %
$F_{h+1}$ is obtained from $F_h$ by  contraction of an edge $e_1$ in series with edge $e_2$ (their common endpoint being incident to 
no other edges). We have $E(F_{h+1})=E(F_h) \backslash \{e_1\}$. Let $m_{h+1}= (d+1) m_h$. Re-indexing to keep things simple define, for $1\le m\le m_h$ and $0\le t\le d$, 
$$
w_{h+1,m}^t(e)=
\begin{cases}
A_t & \text{if $e=e_2$,}\\
w_{h,m}(e) & \text{otherwise}
\end{cases}
$$
and $\alpha_{h+1,m}^t = p_{rs}^t \alpha_{h,m}$. Then, applying Rule \textsf{sr1},  we have the equation
\begin{equation}
\label{Eqn:series}
\begin{tikzpicture}[baseline=(A1),black,node distance=0.5cm,  
solidvert/.style={draw, circle,  fill=red, inner sep=1.8pt},
hollowvert/.style={  draw,  circle,  fill=white,  inner sep=1.8pt},
  every loop/.style={min distance=40pt,in=-30,out=90,looseness=20}]
\node[hollowvert] (A1) at (-1,0) {};
\node[hollowvert] (A2) at (0,0) {};
\node[hollowvert] (A3) at (1,0) {};
\path[-] (A1) edge  [above] node [] {\scriptsize $A_r$} (A2);     
\path[-] (A2) edge [above] node [] {\scriptsize $A_s$} (A3);
\draw (A1) -- (-1.2,0.15); \draw (A1) -- (-1.3,0);  \draw (A1) -- (-1.2,-0.15);
\draw (A3) -- (1.2,0.15); \draw (A3) -- (1.3,0);  \draw (A3) -- (1.2,-0.15);
\draw (-2,0) node[anchor=west]{$\sS( \hspace{3cm}, \emptyset; w_{h,j}) =$};
\def\xsh {0.2}
\draw (3.25+\xsh,0) node[anchor=west]{$\displaystyle \sum_{t=0}^d p_{r,s}^t \  \sS( \hspace{2cm}, \emptyset; w_{h+1,j}^t).$};
\node[hollowvert] (B1) at (5.75+\xsh,0) {};
\node[hollowvert] (B2) at (6.75+\xsh,0) {};
\path[-] (B1) edge [above] node [] {\scriptsize $A_t$} (B2);
\draw (B1) -- (5.55+\xsh,0.15); \draw (B1) -- (5.45+\xsh,0);  \draw (B1) -- (5.55+\xsh,-0.15);
\draw (B2) -- (6.95+\xsh,0.15); \draw (B2) -- (7.05+\xsh,0);  \draw (B2) -- (6.95+\xsh,-0.15);
\end{tikzpicture}
\end{equation}
Summing over $1\le m\le m_h$ gives us
$$ \sum_{m=1}^{m_h} \alpha_{h,m} \ \sS(F_h, \emptyset; w_{h,m}) = 
\sum_{m=1}^{m_h} \sum_{t=0}^d  \alpha_{h+1,m}^t \ \sS(F_{h+1}, \emptyset; w_{h+1,m}^t) $$
and similarly for edge weights in $\BMA'$:
$$ \sum_{m=1}^{m_h} \alpha_{h,m} \ \sS(F_h, \emptyset; \kappa \bullet w_{h,m}) = 
\sum_{m=1}^{m_h} \sum_{t=0}^d  \alpha_{h+1,m}^t \ \sS(F_{h+1}, \emptyset; \kappa \bullet w_{h+1,m}^t) $$
using Proposition \ref{Prop:same}.

\medskip
\noindent
\textsf{parallel}: 
$F_{h+1}$ is obtained from $F_h$ by  deletion of an edge $e_1$ which is in parallel to edge $e_2$. 
 Let $m_{h+1}= m_h$ and, for each $1\le m\le m_h$,  define 
 $w_{h+1,m}: E(F_{h+1}) \rightarrow \BMA$  to be the restriction 
of $w_{h,m}$ to $E(F_h)\backslash \{e_1\}$ and, assuming $w_{h,m}(e_1)=A_r$ and $w_{h,m}(e_2)=A_s$, set $\alpha_{h+1,m} = \delta_{r,s} \alpha_{h,m}$.
Then we have, using Rule \textsf{sr1'}, 
\begin{equation}
\label{Eqn:parallel}
\begin{tikzpicture}[baseline=(A1),black,node distance=0.5cm,  
solidvert/.style={draw, circle,  fill=red, inner sep=1.8pt},
hollowvert/.style={  draw,  circle,  fill=white,  inner sep=1.8pt},
  every loop/.style={min distance=40pt,in=-30,out=90,looseness=20}]
\node[hollowvert] (A1) at (0,0) {};
\node[hollowvert] (A2) at (1,0) {};
\draw (A1) -- (-0.2,0.15); \draw (A1) -- (-0.3,0);  \draw (A1) -- (-0.2,-0.15);
\draw (A2) -- (1.2,0.15); \draw (A2) -- (1.3,0);  \draw (A2) -- (1.2,-0.15);
\path[-,bend left] (A1) edge [above] node  [] {\scriptsize $A_r$} (A2);    
\path[-,bend right] (A1) edge [below] node  [] {\scriptsize $A_s$} (A2);     
\draw (-1,0) node[anchor=west]{$\sS( \hspace{2cm}, \emptyset; w_{h,j}) =$};
\def\xsh {0.2}
\draw (3.2+\xsh,0) node[anchor=west]{$\delta_{r,s} \  \sS( \hspace{2cm}, \emptyset; w_{h+1,j})$};
\node[hollowvert] (B1) at (5+\xsh,0) {};
\node[hollowvert] (B2) at (6+\xsh,0) {};
\path[-] (B1) edge [above] node  [] {\scriptsize $A_s$} (B2);
\draw (B1) -- (4.8+\xsh,0.15); \draw (B1) -- (4.7+\xsh,0);  \draw (B1) -- (4.8+\xsh,-0.15);
\draw (B2) -- (6.2+\xsh,0.15); \draw (B2) -- (6.3+\xsh,0);  \draw (B2) -- (6.2+\xsh,-0.15);
\end{tikzpicture}
\end{equation}
so that 
$$ \sum_{m=1}^{m_h} \alpha_{h,m} \ \sS(F_h, \emptyset; w_{h,m}) = 
\sum_{m=1}^{m_h}  \alpha_{h+1,m} \ \sS(F_{h+1}, \emptyset; w_{h+1,m}) $$
and the same holds with edge weights replaced by $\kappa \bullet w_{h,m}$ and $\kappa \bullet w_{h+1,m}$ by Proposition \ref{Prop:same}.

\medskip
\noindent
\textsf{Delta}:
$F_{h+1}$ is obtained from $F_h$ by  replacing a Delta (the edges of a triangle) with a Wye (a new vertex adjacent only to the three vertices of that triangle). Let us treat $E(F_{h+1})$ as equal to $E(F_h)$ with the understanding that the three edges $e_1$, $e_2$, $e_3$ of the triangle are now the three edges of the Wye, with the convention that $e_u$ in the second graph is incident to neither end of $e_u$ in the first.

Fix an $m$, $1\le m\le m_h$ and set  $A_i=w_{h,m}(e_1)$, $A_j=w_{h,m}(e_2)$, and $A_k=w_{h,m}(e_3)$. Equation~(\ref{Eqn:D-Y}) can be written
\begin{equation}
\label{Eqn:newD-Y}
\sum_{x,y,z\in X} (A_i)_{x,z}(A_j)_{x,y}(A_k)_{y,z} \ \hat{x} \otimes \hat{y} \otimes \hat{z}
=  \sum_{r,s,t=0}^d \rho_{r,s,t}^{i,j,k}\sum_{w,x,y,z\in X}  (A_r)_{w,x}(A_s)_{w,y}(A_t)_{w,z} \ \hat{x} \otimes \hat{y} \otimes \hat{z}
\end{equation}
where $\rho_{r,s,t}^{i,j,k} =  \frac{1}{|X|^3} \displaystyle{\sum_{q_{ab}^c > 0}} \sigma^{i,j,k}_{a, b, c }  Q_{ra} Q_{sb}Q_{tc}$ for all $0\le i,j,k,r,s,t \le d$.
Hence, for all $x, y, z\in X$, 
\begin{equation}
\label{Eqn:hollowD-Y}
   (A_i)_{x,z}(A_j)_{x,y}(A_k)_{y,z}
=
    \sum_{r,s,t=0}^d \rho^{i,j,k}_{r,s,t} \ \sum_{w\in X}  (A_r)_{w,x}(A_s)_{w,y}(A_t)_{w,z}.
\end{equation}
As above, it will be simpler to allow several indices of summation on the right hand side of Equations (\ref{Eqn:Induction}).
For each $0\le r, s, t\le d$, define 
$w_{h+1,m}^{r,s,t}: E(F_{h+1}) \rightarrow \BMA$ via
\begin{equation*}
    w_{h+1,m}^{r,s,t}(e) = \begin{cases}
        A_s &\text{if} \ e=e_1 \cr
        A_t &\text{if} \ e=e_2 \cr
        A_r &\text{if} \ e=e_3 \cr
        w_{h,m}(e) &\text{otherwise.} \end{cases}
\end{equation*}
It follows from (\ref{Eqn:hollowD-Y}) (cf.\ \cite[Proposition 1.5]{WJMscaff}) that
\begin{equation}
\label{Eqn:Delta}
\sS \left( \inlineWhiskeredDelta{$A_i$}{$A_j$}{$A_k$} \ , \ \emptyset \ ; \ w_{h,m} \ \right) = \sum_{r,s,t=0}^d  \rho^{i,j,k}_{r,s,t}  \ \sS \left( \inlineWhiskeredWye{$A_r$}{$A_s$}{$A_t$} \ , \ \emptyset \ ; \ w_{h+1,m}^{r,s,t} \ \right).
\end{equation}
Using Proposition \ref{Prop:same}, we likewise obtain
\begin{equation}
\label{Eqn:Deltakappa}
\sS \left( \inlineWhiskeredDelta{$A_i$}{$A_j$}{$A_k$} \ , \ \emptyset \ ; \ \kappa \bullet w_{h,m} \ \right) = \sum_{r,s,t=0}^d  \rho^{i,j,k}_{r,s,t}  \ \sS \left( \inlineWhiskeredWye{$A_r$}{$A_s$}{$A_t$} \ , \ \emptyset \ ; \ \kappa \bullet w_{h+1,m}^{r,s,t} \ \right).
\end{equation}
Summing over $m$ with coefficients $\alpha_{h+1,m}^{r,s,t}= \rho^{i,j,k}_{r,s,t} \alpha_{h,m}$, we obtain our induction step for the Delta-Wye transformation.

\medskip
\noindent
\textsf{Wye}:
$F_{h+1}$ is obtained from $F_h$ by  replacing a Wye with a Delta. We employ the same conventions as in the previous case.

Fix $m$ and locate those $i$, $j$, $k$ for which  $A_i=w_{h,m}(e_1)$, $A_j=w_{h,m}(e_2)$, and $A_k=w_{h,m}(e_3)$.

With $\varrho_{r,s,t}^{i,j,k} = \sum_{a,b,c=0}^d P_{ai}P_{bj}P_{ck} \tau_{r,s,t}^{a,b,c}$, 
Equation~(\ref{Eqn:Y-D}) can be written
\begin{equation}
\label{Eqn:newY-D}
\sum_{w,x,y,z\in X} (A_i)_{w,x}(A_j)_{w,y}(A_k)_{w,z} \ \hat{x} \otimes \hat{y} \otimes \hat{z}
= \sum_{p_{rs}^t > 0} \varrho_{r,s,t}^{i,j,k}  (A_r)_{x,z}(A_s)_{x,y}(A_t)_{y,z} \ \hat{x} \otimes \hat{y} \otimes \hat{z}.
\end{equation}
Hence, for all $x, y, z\in X$, 
\begin{equation}
\label{Eqn:hollowY-D}
\sum_{w\in X} (A_i)_{w,x}(A_j)_{w,y}(A_k)_{w,z}
= \sum_{p_{rs}^t > 0}  \varrho_{r,s,t}^{i,j,k}\  (A_r)_{x,z}(A_s)_{x,y}(A_t)_{y,z}.
\end{equation}
For each $0\le r, s, t\le d$, define 
$w_{h+1,m}^{r,s,t} :E(F_{h+1}) \rightarrow \BMA$ via
\begin{equation*}
    w_{h+1,m}^{r,s,t}(e) = \begin{cases}
        A_t &\text{if} \ e=e_1 \cr
        A_r &\text{if} \ e=e_2 \cr
        A_s &\text{if} \ e=e_3 \cr
        w_{h,m}(e) &\text{otherwise} \end{cases}
\end{equation*}
and everything goes through as in the previous case, giving
\begin{equation}
\label{Eqn:Wye}
\sS \left( \inlineWhiskeredWye{$A_i$}{$A_j$}{$A_k$} \ , \ \emptyset \ ; \ w_{h,m} \ \right) = \sum_{r,s,t=0}^d  \varrho^{i,j,k}_{r,s,t}  \ \sS \left( \inlineWhiskeredDelta{$A_r$}{$A_s$}{$A_t$} \ , \ \emptyset \ ; \ w_{h+1,m}^{r,s,t} \ \right)
\end{equation}
and
\begin{equation}
\label{Eqn:Wyekappa}
\sS \left( \inlineWhiskeredWye{$A_i$}{$A_j$}{$A_k$} \ , \ \emptyset \ ; \ \kappa \bullet w_{h,m} \ \right) = \sum_{r,s,t=0}^d  \varrho^{i,j,k}_{r,s,t}  \ \sS \left( \inlineWhiskeredDelta{$A_r$}{$A_s$}{$A_t$} \ , \ \emptyset \ ; \ \kappa \bullet w_{h+1,m}^{r,s,t} \ \right).
\end{equation}
Summing over $m$ and using $\alpha_{h+1,m}^{r,s,t} =  \varrho_{r,s,t}^{i,j,k}\alpha_{h,m}$, we establish our induction step. \qed
\end{proof}

We are ready to present our main theorem. By ``a graph in'' an association scheme $(X,\cR)$ we mean  a graph with vertex set $X$ whose adjacency relation is a union of non-identity basis relations from $\cR$.

\begin{theorem}
\label{Thm:quantumisographs}
Let $(X, \cR)$ and $(Y, \cS)$ be exactly triply regular  symmetric association schemes that have the same Delta-Wye parameters.
Let $G'$ be a graph in $(Y, \cS)$ corresponding to $G$ in $(X,\cR)$.   Then $G$ and $G'$ are quantum isomorphic.
\end{theorem}
\begin{proof}
Given any planar graph $F$, we define functions $w$ and $w'$ such that $w(e)=A_G$ and $w'(e) = A_{G'}$ for all $e\in E(F)$.  
By Theorem~\ref{Thm:EqualScaffolds}, we have
$$
\hom(F,G)=\sS(F, \emptyset; w) = \sS(F, \emptyset; w')=\hom(F,G').
$$
By Theorem~\ref{Thm:quantum}, we conclude that $G$ and $G'$ are quantum isomorphic.
\qed
\end{proof}

\section{Hadamard graphs}
\label{Sec:Hadamard}

Hadamard graphs are very closely related to Hadamard matrices, which have received much attention, and are themselves a well-studied family of distance-regular graphs. Our goal in this section is to show that Theorem \ref{Thm:quantumisographs} applies to any pair of Hadamard graphs of the same order.

\subsection{Quantum isomorphism}
\label{Subsec:Hadqc}

A Hadamard matrix is an $n\times n$ $\pm 1$ matrix $H$ satisfying 
$$
HH^\top = nI.
$$
Given an $n\times n$ Hadamard matrix, we construct a graph $G$
on vertex set 
$$ X = \{ r_1^+, r_1^-, \ldots, r_n^+, r_n^-, \ c_1^+, c_1^-, \ldots, c_n^+,c_n^- \},$$
where $r^+_i$ is adjacent to $c^+_j$ and $r^-_i$ is adjacent to $c^-_j$ if $H_{ij}=1$, $r^+_i$ is adjacent to $c^-_j$ and $r^-_i$ is adjacent to $c^+_j$ if $H_{ij}=-1$.
This graph, called \emph{a Hadamard graph of order $4n$}, is a distance regular graph of diameter four with intersection array
$$\{ n,n-1,\frac{n}{2},1; \ 1,\frac{n}{2}, n-1,n \}.$$
For $j=0,\ldots, 4$, we use $A_j$ to denote the $j$-th distance matrix of 
$G$. Then $A_0, \ldots, A_4$ are the adjacency matrices of a 4-class symmetric association scheme.    The matrix of eigenvalues of this association scheme is
$$
P=
\left[\begin{array}{ccccc}
1&n& 2n-2&n & 1\\
1&\sqrt{n}&0&-\sqrt{n}&-1\\
1&0&-2&0&1\\
1&-\sqrt{n}&0&\sqrt{n}&-1\\
1&-n& 2n-2&-n & 1
\end{array}\right].
$$
Since $P^2=4nI$, this association scheme is formally self-dual which means $P_{ij} =Q_{ij}$ and $p_{ij}^k=q_{ij}^k$, for $i,j,k=0,\ldots, 4$.
The intersection numbers are given in  $L_i=[p_{ij}^k]_{k,j}$ with $L_0=I$ and $L_1, \ldots, L_4$ listed in order as
$$\left[ \begin{array}{ccccc} 
\scriptstyle{0}  & \scriptstyle{n}  & \scriptstyle{0}  & \scriptstyle{0}  & \scriptstyle{0}  \\
\scriptstyle{1}  & \scriptstyle{0}  & \scriptstyle{n-1}  & \scriptstyle{0}  & \scriptstyle{0}  \\
\scriptstyle{0}  & \scriptstyle{n/2}  & \scriptstyle{0}  & \scriptstyle{n/2}  & \scriptstyle{0}  \\
\scriptstyle{0}  & \scriptstyle{0}  & \scriptstyle{n-1}  & \scriptstyle{0}  & \scriptstyle{1}  \\
\scriptstyle{0}  & \scriptstyle{0}  & \scriptstyle{0}  & \scriptstyle{n}  & \scriptstyle{0}      \end{array} \right], 
\left[ \begin{array}{ccccc} 
\scriptstyle{0}  & \scriptstyle{0}  & \scriptstyle{2n-2}  & \scriptstyle{0}  & \scriptstyle{0}  \\
\scriptstyle{0}  & \scriptstyle{n-1}  & \scriptstyle{0}  & \scriptstyle{n-1}  & \scriptstyle{0}  \\
\scriptstyle{1}  & \scriptstyle{0}  & \scriptstyle{2n-4}  & \scriptstyle{0}  & \scriptstyle{1}  \\
\scriptstyle{0}  & \scriptstyle{n-1}  & \scriptstyle{0}  & \scriptstyle{n-1}  & \scriptstyle{0}  \\
\scriptstyle{0}  & \scriptstyle{0}  & \scriptstyle{2n-2}  & \scriptstyle{0}  & \scriptstyle{0}      \end{array} \right], 
\left[ \begin{array}{ccccc} 
\scriptstyle{0}  & \scriptstyle{0}  & \scriptstyle{0}  & \scriptstyle{n}  & \scriptstyle{0}  \\
\scriptstyle{0}  & \scriptstyle{0}  & \scriptstyle{n-1}  & \scriptstyle{0}  & \scriptstyle{1}  \\
\scriptstyle{0}  & \scriptstyle{n/2}  & \scriptstyle{0}  & \scriptstyle{n/2}  & \scriptstyle{0}  \\
\scriptstyle{1}  & \scriptstyle{0}  & \scriptstyle{n-1}  & \scriptstyle{0}  & \scriptstyle{0}  \\
\scriptstyle{0}  & \scriptstyle{n}  & \scriptstyle{0}  & \scriptstyle{0}  & \scriptstyle{0}      \end{array} \right],  
\left[ \begin{array}{ccccc} 
\scriptstyle{0}  & \scriptstyle{0}  & \scriptstyle{0}  & \scriptstyle{0}  & \scriptstyle{1}  \\
\scriptstyle{0}  & \scriptstyle{0}  & \scriptstyle{0}  & \scriptstyle{1}  & \scriptstyle{0}  \\
\scriptstyle{0}  & \scriptstyle{0}  & \scriptstyle{1}  & \scriptstyle{0}  & \scriptstyle{0}  \\
\scriptstyle{0}  & \scriptstyle{1}  & \scriptstyle{0}  & \scriptstyle{0}  & \scriptstyle{0}  \\
\scriptstyle{1}  & \scriptstyle{0}  & \scriptstyle{0}  & \scriptstyle{0}  & \scriptstyle{0}      \end{array} \right].$$
We see that $N_p=N_q=35$.

In Nomura's construction of spin models from Hadamard graphs \cite{nomura}, he proves that the association scheme of Hadamard graphs are triply regular
by computing all parameters  $\upsilon^{i,j,k}_{r,s,t}$ satisfying
$$ \inlineSWye{W}{$A_i$,$A_j$,$A_k$}{1} = \sum_{r,s,t} \upsilon^{i,j,k}_{r,s,t}  \inlineSDelta{D}{$A_r$,$A_s$,$A_t$}{1}.$$
By  Lemma~\ref{Lem:exactlytrip}, the association scheme of a Hadamard graph is exactly triply regular.

Further, all of the parameters $\upsilon^{i,j,k}_{r,s,t}$ depend only on $n$.
For instance, when $i=j=k=1$, 
\begin{equation}
\label{Eqn:upsilon}
    \upsilon_{0,0,0}^{1,1,1}=n, \quad \upsilon_{0,2,2}^{1,1,1}=\upsilon_{2,0,2}^{1,1,1}=\upsilon_{2,2,0}^{1,1,1}=\frac{n}{2}, \quad
\upsilon_{2,2,2}^{1,1,1}=\frac{n}{4},
\end{equation}
and
$\upsilon_{r,s,t}^{1,1,1}=0$ for all other $r$, $s$ and $t$.
Since 
$$ \inlineSWye{W}{$A_i$,$A_j$,$A_k$}{1} = \sum_{r,s,t} \upsilon^{i,j,k}_{r,s,t}  \sum_{a,b,c} P_{ar}P_{bs}P_{ct} \inlineSDelta{D}{$E_a$,$E_b$,$E_c$}{1},$$ 
and the coefficients on the right-hand side depend only on $n$,
the association schemes of non-isomorphic Hadamard graphs of order $4n$
have the same Delta-Wye parameters.   
 Our next result follows immediately from Theorem~\ref{Thm:quantumisographs}.

\begin{theorem}
\label{Thm:Hadamard}
Any two Hadamard graphs of order $4n$ are quantum isomorphic.
\qed
\end{theorem}

\begin{remark}
Given a Hadamard graph of order $4n$, let $s, t_0, t_1, t_2, t_3, t_4$ be complex numbers satisfying
$$
s^2+2(2n-1)s+1=0, \quad
t_0^2=\frac{2\sqrt{n}}{(4n-1)s+1}, \quad t_1^4=1, \quad
t_2=st_0,\quad t_3=-t_1
\quad\text{and}\quad t_4=t_0.
$$
Then the matrices $W_+=\sum_{j=0}^4 t_j A_j$ and $W_-=\sum_{j=0}^4 t_j^{-1} A_j$ form a spin model \cite{nomura}.
In \cite{jones},  Jones constructed from each link diagram a plane graph $F$ with signed edges and showed that the scaffold
$\sS(F, \emptyset; w)$, where 
$$
w(e) = \begin{cases} W_+ & \text{if the sign of $e$ is $+$}\\ W_- & \text{if the sign of $e$ is $-$}\end{cases},
$$
is a link invariant with some simple normalization.

Using Theorem~\ref{Thm:EqualScaffolds}, we reproduce Jaeger's proof that the spin models from any Hadamard graphs of the same order give the same link invariant \cite[Proposition 22]{jaeger}.

\end{remark}

\subsection{Examples of homomorphism counts}
\label{Subsec:examples}

We have seen that scaffolds of order zero include homomorphism counts and we
have given a general recipe for reducing these scaffolds to sums of single-vertex scaffolds.  Here, we give two concrete examples where we compute $\hom(F,G)$ where $F$ is a small planar graph and $G$ is a Hadamard graph on $4n$ vertices. In the first example, $F$ is a series-parallel graph and the Delta-Wye equations are not needed.


\begin{example}
The number of homomorphisms from the complete bipartite graph $K_{2,3}$ into a Hadamard graph $G$ on $4n$
vertices with adjacency matrix $A=A_1$ is  $\hom(K_{2,3} ,G)=n^4(n+3)$, the sum of entries of the matrix $(A_1^2) \circ (A_1^2) \circ (A_1^2)$
using the expansion $A_1^2 = nA_0 + \frac{n}{2} A_2$:
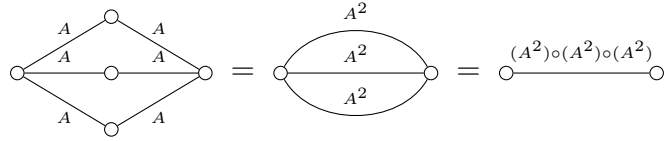
\begin{center}
\begin{tikzpicture}
[black,hollownode/.style={draw, circle, fill=white, inner sep=1.8pt},rootnode/.style={draw, circle, fill=red, inner sep=1.8pt},
      every loop/.style={min distance=40pt,in=-30,out=90,looseness=20}]
\node[hollownode] (A1) at (0,0) {};
\node[hollownode] (B1) at (1.25,-0.75) {};
\node[hollownode] (B2) at (1.25,0) {};
\node[hollownode] (B3) at (1.25,0.75) {};
 \node[hollownode] (A2) at (2.5,0) {};
\path (A1) edge node[below] {$\scriptscriptstyle{A}$} (B1);  
\path (A1) edge node[above] {$\scriptscriptstyle{A}$} (B2); 
\path (A1) edge node[above] {$\scriptscriptstyle{A}$} (B3); 
\path (A2) edge node[below] {$\scriptscriptstyle{A}$} (B1); 
\path (A2) edge node[above] {$\scriptscriptstyle{A}$} (B2); 
\path (A2) edge node[above] {$\scriptscriptstyle{A}$} (B3); 
\node (eq1) at (3,0) {$=$};
\def\xsh {3.5}
\node[hollownode] (C1) at (\xsh+0,0) {};
\node[hollownode] (C2) at (\xsh+2,0) {};
\path[bend right=60] (C1) edge node[above] {$\scriptscriptstyle{A^2}$} (C2);  
\path (C1) edge node[above] {$\scriptscriptstyle{A^2}$} (C2); 
\path[bend left=60] (C1) edge node[above] {$\scriptscriptstyle{A^2}$} (C2); 
\node (eq2) at (\xsh+2.5,0) {$=$};
\def\xxsh {6.5}
\node[hollownode] (D1) at (\xxsh+0,0) {};
\node[hollownode] (D2) at (\xxsh+2,0) {};  
\path (D1) edge node[above] {$\scriptscriptstyle{(A^2)\circ (A^2)\circ (A^2)}$} (D2); 
\end{tikzpicture}
\end{center}
\end{example}

Now let's work with a more complicated example where the extended series-parallel reduction rules are applied in conjunction with the scaffold expansion afforded by triple regularity.

\begin{example}
Consider the graph $F$, below, obtained by taking a 1-clique sum 
of a 4-cycle and the graph obtained by deleting one vertex of the 3-cube. Let $A$ be the adjacency matrix of a 
Hadamard graph $G$ on $4n$ vertices with Bose-Mesner algebra having basis of 
Schur idempotents $\{A_0=I,A_1=A,A_2,A_3,A_4\}$ ordered according to distance in $G$. We compute, using scaffold rules \textsf{sr1} (series reduction) and 
\textsf{sr'} (loop removal),

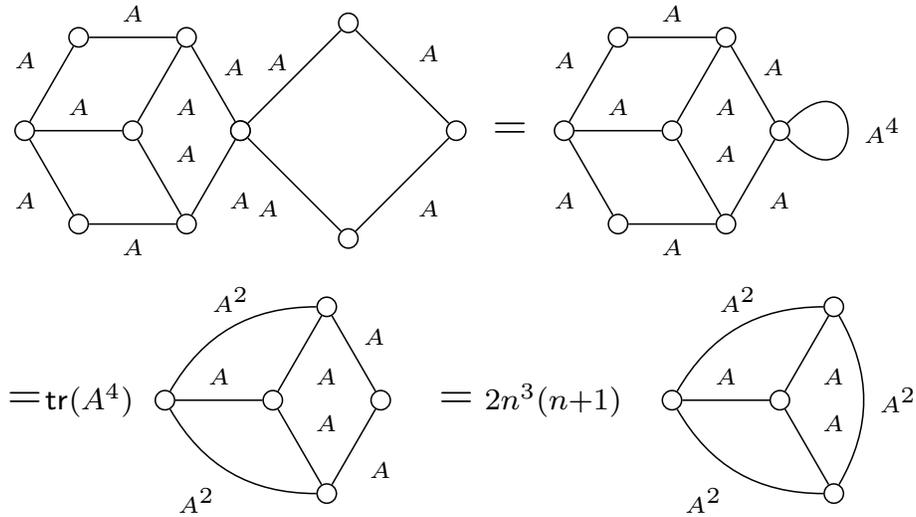
\begin{center}
\resizebox{5in}{!}{
  \begin{tikzpicture}[black,hollownode/.style={draw, circle, fill=white, inner sep=1.8pt},rootnode/.style={draw, circle, fill=red, inner sep=1.8pt},
      every loop/.style={min distance=40pt,in=-30,out=90,looseness=20}]
      \def\r {1.0}
\begin{scope}[shift={(0,0)}] 
\node[hollownode] (e) at (0:\r) {};   \node[hollownode] (ne) at (60:\r) {};   \node[hollownode] (nw) at (120:\r) {};   
\node[hollownode] (w) at (180:\r) {};   \node[hollownode] (sw) at (240:\r) {};   \node[hollownode] (se) at (300:\r) {};   
\node[hollownode] (c) at (0:0) {}; 
 \path (e) edge node[above,pos=0.4] {$\scriptscriptstyle{\ \ A}$} (ne);  \path (ne) edge node[above] {$\scriptscriptstyle{A}$} (nw);  
 \path (nw) edge node[above left] {$\scriptscriptstyle{A}$} (w);
 \path (w) edge node[below left] {$\scriptscriptstyle{A}$} (sw);  \path (sw) edge node[below] {$\scriptscriptstyle{A}$} (se);  
 \path (se) edge node[below right] {$\scriptscriptstyle{A}$} (e); 
 \path (c) edge node[above] {$\scriptscriptstyle{A}$} (w);  \path (c) edge node[below right] {$\scriptscriptstyle{A}$} (ne);  
 \path (c) edge node[above right] {$\scriptscriptstyle{A}$} (se);  
 \end{scope}
 \begin{scope}[shift={(2,0)}] 
\node[hollownode] (lft) at (180:\r) {};   \node[hollownode] (top) at (90:\r) {};   
\node[hollownode] (rigt) at (0:\r) {};   \node[hollownode] (bot) at (270:\r) {};   
 \path (lft) edge node[above,pos=0.4] {$\scriptscriptstyle{A\ }$} (top);  \path (top) edge node[above right] {$\scriptscriptstyle{A}$}  (rigt);  
 \path (rigt) edge node[below right] {$\scriptscriptstyle{A}$} (bot);  \path (bot) edge node[below left] {$\scriptscriptstyle{A}$}   (lft); 
 \end{scope}
\node (eq) at (3.5,0) {$=$};
\begin{scope}[shift={(5,0)}] 
\node[hollownode] (e) at (0:\r) {};   \node[hollownode] (ne) at (60:\r) {};   \node[hollownode] (nw) at (120:\r) {};   
\node[hollownode] (w) at (180:\r) {};   \node[hollownode] (sw) at (240:\r) {};   \node[hollownode] (se) at (300:\r) {};   
\node[hollownode] (c) at (0:0) {}; 
 \path (e) edge node[above,pos=0.4] {$\scriptscriptstyle{\ \ A}$} (ne);  \path (ne) edge node[above] {$\scriptscriptstyle{A}$} (nw);  
 \path (nw) edge node[above left] {$\scriptscriptstyle{A}$} (w);
 \path (w) edge node[below left] {$\scriptscriptstyle{A}$} (sw);  \path (sw) edge node[below] {$\scriptscriptstyle{A}$} (se);  
 \path (se) edge node[below right] {$\scriptscriptstyle{A}$} (e); 
 \path (c) edge node[above] {$\scriptscriptstyle{A}$} (w);  \path (c) edge node[below right] {$\scriptscriptstyle{A}$} (ne);  
 \path (c) edge node[above right] {$\scriptscriptstyle{A}$} (se);  
 \path (e) edge[loop,in=45,out=-45,looseness=20] node[right] {$\scriptscriptstyle{A^4}$}  (e);
 \end{scope}
 \node (eq2) at (-1.0,-2.5) {$=$};
 \node (coeff) at (-0.4,-2.5) {$\scriptstyle{\mathsf{tr}(A^4)}$};
\begin{scope}[shift={(1.3,-2.5)}] 
\node[hollownode] (e) at (0:\r) {};   \node[hollownode] (ne) at (60:\r) {};  
\node[hollownode] (w) at (180:\r) {};    \node[hollownode] (se) at (300:\r) {};   
\node[hollownode] (c) at (0:0) {}; 
 \path (e) edge node[above,pos=0.4] {$\scriptscriptstyle{\ \ A}$} (ne);  \path (ne) edge[bend right] node[above] {$\scriptscriptstyle{A^2}$} (w);  
 \path (w) edge[bend right] node[below left] {$\scriptscriptstyle{A^2}$} (se);  
 \path (se) edge node[below right] {$\scriptscriptstyle{A}$} (e); 
 \path (c) edge node[above] {$\scriptscriptstyle{A}$} (w);  \path (c) edge node[below right] {$\scriptscriptstyle{A}$} (ne);  
 \path (c) edge node[above right] {$\scriptscriptstyle{A}$} (se);  
 \end{scope}
 \node (eq2) at (3.0,-2.5) {$=$};
 \node (coeff) at (3.9,-2.5) {$\scriptstyle{2n^3(n+1)}$};
\begin{scope}[shift={(6.0,-2.5)}] 
\node[hollownode] (ne) at (60:\r) {};   \node[hollownode] (w) at (180:\r) {};    \node[hollownode] (se) at (300:\r) {};   
\node[hollownode] (c) at (0:0) {}; 
 \path (ne) edge[bend right] node[above] {$\scriptscriptstyle{A^2}$} (w);  
 \path (w) edge[bend right] node[below left] {$\scriptscriptstyle{A^2}$} (se);  
 \path (se) edge[bend right] node[right] {$\scriptscriptstyle{A^2}$} (ne);   
 \path (c) edge node[above] {$\scriptscriptstyle{A}$} (w);  \path (c) edge node[below right] {$\scriptscriptstyle{A}$} (ne);  
 \path (c) edge node[above right] {$\scriptscriptstyle{A}$} (se);  
 \end{scope}
    \end{tikzpicture} 
}
\end{center}
As a consequence of exact triple regularity, we have from (\ref{Eqn:upsilon})
\inlineSWye{A111}{$A$,$A$,$A$}{0.7} $=$
$$n  \inlineSDelta{D}{$A_0$,$A_0$,$A_0$}{0.7} 
+  \frac{n}{2}   \inlineSDelta{D}{$A_0$,$A_2$,$A_2$}{0.7} 
+  \frac{n}{2}    \inlineSDelta{D}{$A_2$,$A_0$,$A_2$}{0.7} 
+  \frac{n}{2}    \inlineSDelta{D}{$A_2$,$A_2$,$A_0$}{0.7}  +  \frac{n}{4}    \inlineSDelta{D}{$A_2$,$A_2$,$A_2$}{0.7}. $$
We make this substitution to find

\resizebox{4.5in}{!}{
  \begin{tikzpicture}[black,hollownode/.style={draw, circle, fill=white, inner sep=1.8pt},
  rootnode/.style={draw, circle, fill=red, inner sep=1.8pt},
  every loop/.style={min distance=40pt,in=-30,out=90,looseness=20}]
\def\r {1.0}
 \node (hom) at (-2.3,2) {$\mathsf{hom}(F,G) =$};
 \node (coeff) at (0,2) {$2n^3(n+1)$};
\begin{scope}[shift={(2.3,2)}] 
\node[hollownode] (ne) at (60:\r) {};   \node[hollownode] (w) at (180:\r) {};    \node[hollownode] (se) at (300:\r) {};   
\node[hollownode] (c) at (0:0) {}; 
 \path (ne) edge[bend right] node[above] {$\scriptscriptstyle{A^2}$} (w);  
 \path (w) edge[bend right] node[below left] {$\scriptscriptstyle{A^2}$} (se);  
 \path (se) edge[bend right] node[right] {$\scriptscriptstyle{A^2}$} (ne);   
 \path (c) edge node[above] {$\scriptscriptstyle{A}$} (w);  \path (c) edge node[below right] {$\scriptscriptstyle{A}$} (ne);  
 \path (c) edge node[above right] {$\scriptscriptstyle{A}$} (se);  
 \end{scope}
 \node (nexteq) at (5.2,2) {$ = \ 2n^3(n+1) \ \cdot$};
\end{tikzpicture} 
} 
$$\left[ n \  \inlineBEG{0}{0}{0}{1.0} + \frac{n}{2} \left(  
\inlineBEG{0}{2}{2}{1.0} +  \inlineBEG{2}{0}{2}{1.0} +  \inlineBEG{2}{2}{0}{1.0} \right) 
+ \frac{n}{4}  \inlineBEG{2}{2}{2}{1.0} \right]$$

and, with $A^2\circ A_0=nA_0$ and $A^2 \circ A_2 = \frac{n}{2}A_2$, we apply Rule \textsf{sr1'} (parallel reduction) to arrive at 
\begin{eqnarray*}
&=& 2n^3(n+1) \left[ n^4 \inlineEG{0}{0}{0}{1.0} + \frac{n^4}{8} \left(  
\inlineEG{0}{2}{2}{1.0} +  \inlineEG{2}{0}{2}{1.0} +  \inlineEG{2}{2}{0}{1.0} \right) 
+ \frac{n^4}{32}  \inlineEG{2}{2}{2}{1.0} \right] \\
&=& 2n^3(n+1) \left[  n^4 \tr(I) +  \frac{n^4}{8} \cdot 3 \cdot \tr(A_2^2) +  \frac{n^4}{32} \tr(A_2^3) \right] \\
&=& n^{11} + 4n^{10} + 7n^9 + 4n^8.
\end{eqnarray*}
 \end{example}

\section{Discussion and open problems}
\label{Sec:problems}

In this paper, we have used association schemes as a place to search for graphs where subconfiguration counts are somewhat under control. We have relied on a theorem of Man\v{c}inska and Roberson that shows it is sufficient to count homomorphisms into $G$ from any planar graph $F$ and also on a theorem of Epifanov that gives a reduction procedure for any planar graph $F$ involving moves on just 1, 2, or 3 edges. The algebraic effect of single-edge and 2-edge reductions can be computed in any association scheme, but the moves involving three edges --- the Delta-Wye transformations --- seem only manageable in the case of exactly triply regular association schemes. Fortunately, the association scheme of any Hadamard graph has this property.

In Section~\ref{Subsec:Hadqc}, we use homomorphism counting to show two Hadamard graphs of the same order, $G$ and $H$, are quantum isomorphic.
The immediate question is to determine a quantum permutation matrix $\cP$ of order $d$, for some $d$,
satisfying $\cP(A_G\otimes I_d) = (A_H \otimes I_d)\cP$, which will give a perfect quantum strategy for the non-local isomorphism game.

Schmidt's example and Hadamard graphs are graphs in association schemes.  We are interested in the common properties shared by association schemes that contain quantum isomorphic graphs.  In particular, do they have to be both exactly triply regular with the same Delta-Wye parameters?  We ask for more pairs of exactly triply regular association schemes with the same Delta-Wye parameters, which will give more examples of quantum isomorphic graphs.    A related question is whether the Delta-Wye parameters of an exactly triply regular association scheme are determined by its intersection numbers.

In Section~6 of \cite{manrob4}, the authors mention their first example of quantum isomorphic graphs can be constructed using both their method as well as a version of the Cai, F\"{u}rer and Immerman construction.  The Cai, F\"{u}rer and Immerman construction is designed to produce non-isomorphic graphs that are indistinguishable by the Weisfeiler-Lehman algorithm \cite{cfi}.   
A natural question is whether two non-isomorphic Hadamard graphs of the same order are distinguishable by the $d$-dimensional Weisfeiler-Lehman algorithm, for some $d$.

The association schemes supporting spin models that give the Kauffman polynomial or the Hadamard spin models are formally self-dual and exactly triply regular,
\cite{jaegerSRG} and \cite{nomura}.   Does the Bose-Mesner algebra of a formally self-dual exactly triply regular association scheme always contain a spin model?
Conversely, is the Nomura algebra of a spin model exactly triply regular?  In \cite{jaeger}, Jaeger asked for examples of exactly triply regular association schemes that are not formally self-dual or a proof that such an association scheme cannot exist.

\section*{Acknowledgments}

  This work was supported, in part, through an  NSERC Discovery Grant RGPIN-2021-03609 (AC) and NSF DMS Award \#1808376 (WJM); this essential support is gratefully acknowledged. 

This work began at the  Centre de Recherches Math\'{e}matiques at the Universit\'{e} de Montr\'{e}al during their Workshop on Graph Theory, Algebraic Combinatorics and Mathematical Physics (July 25 -- August 19, 2022). The authors are grateful to the institute for hosting this meeting.   The first author gratefully acknowledges the support of the CRM-Simons program.   The second author gratefully acknowledges travel support provided by the US National Science Foundation (Award \#2212755).  

The authors thank Bill Kantor for helpful discussions about Hadamard graphs,
to Laura Man\v{c}inska and Dave Roberson for introducing the authors to the topic of quantum isomorphisms and answering our questions, and to Simon Schmidt for telling us about his work. P.\ K.\ Aravind provided comments to improve the introduction and Daniel Gromada identified mistakes in an earlier draft.



\end{document}